\documentclass[11pt]{article}
\usepackage{amsthm,amsfonts, amsbsy, amssymb,amsmath,graphicx}
\usepackage{graphics}
\usepackage[english]{babel}
\usepackage{graphicx}
\usepackage{lineno}
\usepackage{enumerate}
\usepackage{tikz}
\usepackage{tkz-graph}
\usepackage{tkz-berge}
\usepackage{hyperref}
\usepackage{color}
\usepackage{tikz}

\hypersetup{colorlinks=true}

\hypersetup{colorlinks=true, linkcolor=blue, citecolor=blue,urlcolor=blue}

\input epsf
\usepackage{epic,latexsym,amssymb}
\usepackage{tikz}

\usepackage{tikz,epsf}

\usepackage{amsfonts}
\usepackage{amscd}
\usepackage{amsmath}
\usepackage{graphicx}
\usepackage{color}
\usepackage{caption,subcaption}

\newtheorem{theorem}{Theorem}

\newtheorem{lemma}[theorem]{Lemma}

\newcommand{\diam}{{\rm diam}}

\tikzstyle{vertex}=[circle, draw, inner sep=0pt, minimum size=6pt]

\newcommand{\QEDmark}{\mbox{\textsc{qed}}}
\newcommand{\proofStarter}[1]{\textsc{#1}}

\begin{document}

\title{Injective coloring of product graphs\vspace{7mm}}
\date{}
\author {
Babak Samadi$^a$\thanks{Corresponding author}, Nasrin Soltankhah$^b$ and Ismael G. Yero$^c$\vspace{1.5mm}\\
$^{a,b}$Department of Mathematics, Faculty of Mathematical Sciences, Alzahra University,\\
Tehran, Iran\\
{\tt b.samadi@alzahra.ac.ir}\\
{soltan@alzahra.ac.ir}\vspace{1.5mm}\\
$^c$Departamento de Matem\'{a}ticas, Universidad de C\'{a}diz, ETSI Algeciras, Spain\\
{\tt ismael.gonzalez@uca.es}\vspace{3mm}\\
}
\date{}
\maketitle

\begin{abstract}
The problem of injective coloring in graphs can be revisited through two different approaches: coloring the two-step graphs and vertex partitioning of graphs into open packing sets, each of which is equivalent to the injective coloring problem itself. Taking these facts into account, we observe that the injective coloring lies between graph coloring and domination theory.

We make use of these three points of view in this paper so as to investigate the injective coloring of some well-known graph products. We bound the injective chromatic number of direct and lexicographic product graphs from below and above. In particular, we completely determine this parameter for the direct product of two cycles. We also give a closed formula for the corona product of two graphs.


\end{abstract}

\vspace{1mm}



\section{Introduction}

Throughout this paper, we consider $G$ as a finite simple graph with vertex set $V(G)$ and edge set $E(G)$. The {\em open neighborhood} of a vertex $v$ is denoted by $N_{G}(v)$, and its {\em closed neighborhood} is $N_{G}[v]=N_{G}(v)\cup \{v\}$. The {\em minimum} and {\em maximum degrees} of $G$ are denoted by $\delta(G)$ and $\Delta(G)$, respectively. Given the subsets $A,B\subseteq V(G)$, by $[A,B]$ we mean the set of edges with one end point in $A$ and the other in $B$. Finally, for a given set $S\subseteq V(G)$, by $G[S]$ we represent the subgraph of $G$ induced by $S$. We use \cite{we} as a reference for terminology and notation which are not explicitly defined here.

\subsection{Main terminology}

For all four standard products of graphs $G$ and $H$ (according to \cite{ImKl}), the vertex set of the product is $V(G)\times V(H)$. Their edge sets are defined as follows.
\begin{itemize}
\item In the \emph{Cartesian product} $G\square H$ two vertices are adjacent if they are adjacent in one coordinate and equal in the other.
\item In the \emph{direct product} $G\times H$ two vertices are adjacent if they are adjacent in both coordinates.
\item The edge set of the \emph{strong product} $G\boxtimes H$ is the union of $E(G\square H)$ and $E(G\times H)$.
\item Two vertices $(g,h)$ and $(g',h')$ are adjacent in the \emph{lexicographic product} $G\circ H$ if either $gg'\in E(G)$ or ``$g=g'$ and $hh'\in E(H)$''.
\end{itemize}
Note that all these four products are associative and only the first three ones are commutative, while the lexicographic product is not (see \cite{ImKl}).

A function $f:V(G)\rightarrow\{1,\dots,k\}$ is an {\em injective $k$-coloring function} if no vertex $v$ is adjacent to two vertices $u$ and $w$ with $f(u)=f(w)$. For such a function $f$, the set of color classes $\big{\{}\{v\in V(G)\mid f(v)=i\}\big{\}}_{1\leq i\leq k}$ is an \emph{injective $k$-coloring} of $G$ (or simply an \textit{injective coloring} if $k$ is clear from the context). The minimum $k$ for which a graph $G$ admits an injective $k$-coloring is the {\em injective chromatic number} $\chi_{i}(G)$ of $G$. Injective colorings were introduced in \cite{hkss}, and further studied in \cite{bsy,jxz,pp,sy} for just some examples.

Another approach to the injective coloring of graphs, which is indeed previous to the idea of injective colorings, can be presented as follows. The {\em two-step graph} $\mathcal{N}(G)$ of a graph $G$ is the graph having the same vertex set as $G$ with an edge joining two vertices in $\mathcal{N}(G)$ if and only if they have a common neighbor in $G$. These graphs were introduced in \cite{av} and investigated later in \cite{bd}, \cite{eh} and \cite{lr}. Since a vertex subset $S$ is independent in $\mathcal{N}(G)$ if and only if every two vertices of $S$ have no common neighbor in $G$, we readily observe that
\begin{equation}\label{open-previous}
\chi_{i}(G)=\chi\big{(}\mathcal{N}(G)\big{)},
\end{equation}
in which $\chi$ is the well-known chromatic number. This exposition is centered into giving some contributions to the injective coloring of graph products.

\subsection{Related concepts and plan of the article}

This subsection is devoted to show some strong relationships between the above-mentioned concept and other ones related to domination theory. A subset $S\subseteq V(G)$ is a {\em dominating set} (resp. \textit{total dominating set}) if each vertex in $V(G)\backslash S$ \big{(}resp. $V(G)$\big{)} has at least one neighbor in $S$. The {\em domination number} $\gamma(G)$ \big{(}resp. \textit{total domination number} $\gamma_{t}(G)$\big{)} is the minimum cardinality among all dominating sets (resp. total dominating sets) in $G$. For more information on domination theory, the reader can consult \cite{hhh,hhh-1}.

The study of distance coloring of graphs was initiated by Kramer and Kramer (\cite{kk2} and \cite{kk1}) in $1969$. A {\em $2$-distance coloring} (or, $2$DC for short) of a graph $G$ is a mapping of $V(G)$ to a set of colors (nonnegative integers) by which any two vertices at distance at most two receive different colors. The minimum number of colors $k$ for which there is a $2$DC of $G$ is called the {\em $2$-distance chromatic number} $\chi_{2}(G)$ of $G$.

By a $\chi_{i}(G)$-coloring and a $\chi_{2}(G)$-coloring we mean an injective coloring and a $2$DC of $G$ of cardinality $\chi_{i}(G)$ and $\chi_{2}(G)$, respectively.

On the other hand, problems regarding vertex partitioning are classical in graph theory. In fact, there are many different ways for such partitioning into sets satisfying a specific property. For instance, when dealing with ``domination'' (resp. ``total domination"), the problem of finding the maximum cardinality of a vertex partition of a graph $G$ into dominating sets (resp. total dominating sets) has been widely investigated in the literature. The study of the associated parameter, called {\em domatic number} $d(G)$ \big{(}resp. \textit{total domatic number} $d_{t}(G)$\big{)}, was first carried out in \cite{ch} (resp. \cite{cdh}), see also the books \cite{hhh-2,hhh} . Also, a topic connected to domination is that one of packings. Based on such close relationships, one would find interesting to consider graph partitioning problems regarding packing sets. In this concern, a subset $B\subseteq V(G)$ is called a {\em packing} (or a {\em packing set}) in $G$ if for each distinct vertices $u,v\in B$, $N_{G}[u]\cap N_{G}[v]=\emptyset$ (equivalently, $B$ is a packing in $G$ if $|N_{G}[v]\cap B|\leq1$ for all $v\in B$). The {\em packing number} $\rho(G)$ is the maximum cardinality among all packing sets in $G$. In connection with this, a vertex partition $\mathbb{P}=\{P_{1},\dots,P_{|\mathbb{P}|}\}$ of a graph $G$ is called a {\em packing partition} if $P_{i}$ is a packing in $G$ for each $1\leq i\leq|\mathbb{P}|$. The {\em packing partition number} $p(G)$ is the minimum cardinality among all such partitions of $G$.

In contrast with the construction of $\mathcal{N}(G)$ for a graph $G$, the {\em closed neighborhood graph} $\mathcal{N}_{c}(G)$ of a graph $G$ has vertex set $V(G)$, and two distinct vertices $u$ and $v$ are adjacent in $\mathcal{N}_{c}(G)$ if and only if $N_{G}[u]\cap N_{G}[v]\neq \emptyset$ (see \cite{bkr}). With this in mind, we observe that a $2$DC of a graph $G$ is the same as a coloring of its closed neighborhood graph. It is also easily seen that $\mathcal{N}_{c}(G)$ is isomorphic to the {\em square} $G^{2}$ of $G$. We, in addition, note that the $2$DC problem is equivalent to the problem of vertex partitioning of a graph into packings. Therefore, we altogether observe that
\begin{equation}\label{closed}
\chi_{2}(G)=\chi(G^{2})=\chi\big{(}\mathcal{N}_{c}(G)\big{)}=p(G).
\end{equation}

A natural situation, concerning domination and packings, comes across while considering the study of related parameters in which open neighborhoods are used, instead of closed neighborhoods. Indeed, a subset $B\subseteq V(G)$ is said to be an {\em open packing} (or an {\em open packing set}) in $G$ if for any distinct vertices $u,v\in B$, $N_{G}(u)\cap N_{G}(v)=\emptyset$. The {\em open packing number}, denoted by $\rho_{o}(G)$, is the maximum cardinality among all open packing sets in $G$ (see \cite{hs}).

Motivated by the existence of packing partitions and open packing sets, we might say that a vertex partition $\mathbb{P}=\{P_{1},\dots,P_{|\mathbb{P}|}\}$ of a graph $G$ is an {\em open packing partition} (OPP for short) if $P_{i}$ is an open packing in $G$ for each $1\leq i\leq|\mathbb{P}|$. The {\em open packing partition number} $p_{o}(G)$ is the minimum cardinality among all OPPs of $G$. In this paper, we investigate this kind of vertex partitioning of graphs. However, it turns out that such partitions can be considered from other approaches since we readily observe that for any $k$-injective coloring function $f$, the partition $\big{\{}\{v\in V(G)\mid f(v)=i\}\big{\}}_{1\leq i\leq k}$ forms an OPP of $G$ and vice versa. This fact, together with \eqref{open-previous} and the idea of two-step graphs, leads to
\begin{equation}\label{open}
\chi_{i}(G)=\chi\big{(}\mathcal{N}(G)\big{)}=p_{o}(G),
\end{equation}
which is an open analogue of (\ref{closed}). This establishes another relationship between graph coloring and domination theory. In connection with this, we next give some contributions to the injective chromatic number (or OPP number) of graphs (or equivalently, to the chromatic number of two-step graphs) with emphasis on graph products. In this sense, from now on, we indistinctively use the three terminologies in concordance with the best use in each situation.

This paper is organized as follows. We consider the direct, lexicographic and corona products in order to study $\chi_{i}(G*H)$ in general case, in which $*\in\{\times,\circ,\odot\}$ (here $\odot$ represents the corona product, which is not exactly a standard product as defined in \cite{ImKl}, but it can be taken more as a graph operation). Sharp lower and upper bounds are exhibited on the injective chromatic number when dealing with the direct and lexicographic product graphs. In the case of corona product graphs, we give a closed formula for this parameter and prove that it assumes all values given in the formula. In particular, a counterexample to a formula given in \cite{grt} concerning the $2$-distance chromatic number of lexicographic product graphs is presented. Moreover, when dealing with the direct product graphs, we completely determine the injective chromatic number of direct product of two cycles by using the new tools given in this paper and some classical results in the literature.


\section{Direct product graphs}

Let $u\in V(G)$ be a vertex of maximum degree and let $\mathbb{B}=\{B_{1},\dots,B_{\chi_{i}(G)}\}$ be a $\chi_{i}(G)$-coloring. Because $B_{j}$ is an open packing in $G$ for each $1\leq j\leq \Delta(G)$, $u$ has at most one neighbor in $B_{j}$ and hence $\sum_{j=1}^{\chi_{i}(G)}|N(u)\cap B_{j}|\leq \chi_{i}(G)$. So,
\begin{equation}\label{EQ1}
\chi_{i}(G)\geq \Delta(G).
\end{equation}
This simple but important inequality will turn out to be useful in some places in this paper.


\subsection{General case}

We here bound the injective chromatic number of direct product graphs from above and below. To this end, given two graphs $G$ and $H$ defined on the same vertex set, by $G\sqcup H$ we mean the graph with vertex set $V(G\sqcup H)=V(G)=V(H)$ and edge set $E(G\sqcup H)=E(G)\cup E(H)$. We assume that $I_{G}$ and $I_{H}$ are the sets of isolated vertices of $G$ and $H$, respectively. Let $G^{-}$ and $H^{-}$ be the graphs obtained from $G$ and $H$ by removing all isolated vertices from $G$ and $H$, respectively. Moreover, by $G+H$ we represent the disjoint union of two graphs $G$ and $H$.

We proceed with the following lemma which will have important roles throughout this section.

\begin{lemma}\label{Iso}
Let $G$ and $H$ be any graphs. Then,
$$\mathcal{N}(G\times H)\cong \big{(}\mathcal{N}(G^{-})\boxtimes \mathcal{N}(H^{-})\big{)}^{-}+\overline{K_{p}},$$
in which $p$ is the cardinality of
$$I_{\mathcal{N}(G\times H)}=\{(g,h)\mid \mbox{for each $(g',h')$, $N_{G}(g)\cap N_{G}(g')=\emptyset$ or $N_{H}(h)\cap N_{H}(h')=\emptyset$}\}.$$
\end{lemma}

\begin{proof}
It is readily observed that
$$\{(g,h)\mid \mbox{for each $(g',h')$, $N_{G}(g)\cap N_{G}(g')=\emptyset$ or $N_{H}(h)\cap N_{H}(h')=\emptyset$}\}$$
is the set of isolated vertices of $\mathcal{N}(G\times H)$. We next recall that
$$\mathcal{N}(G^{-})\boxtimes \mathcal{N}(H^{-})=\big{(}\mathcal{N}(G^{-})\times \mathcal{N}(H^{-})\big{)}\sqcup \big{(}\mathcal{N}(G^{-})\square \mathcal{N}(H^{-})\big{)}.$$

Let $(g,h)(g',h')$ be an edge in $\mathcal{N}(G\times H)^{-}$. By symmetry, we may assume that $h\neq h'$. Therefore, there exists a vertex $(g'',h'')$ adjacent to both $(g,h)$ and $(g',h')$ in $G\times H$. This means that $g''g,g''g'\in E(G)$ and $h''h,h''h'\in E(H)$. In particular, we have $g,g',g''$ and $h,h',h''$ are vertices of $G^{-}$ and $H^{-}$, respectively. Therefore, both $N_{G^{-}}(g)\cap N_{G^{-}}(g')$ and $N_{H^{-}}(h)\cap N_{H^{-}}(h')$ are nonempty. We now have ``$gg'\in E\big{(}\mathcal{N}(G^{-})\big{)}$ and $hh'\in E\big{(}\mathcal{N}(H^{-})\big{)}$" if $g\neq g'$, and ``$g=g'$ and $hh'\in E\big{(}\mathcal{N}(H^{-})\big{)}$" otherwise. Consequently, $(g,h)(g',h')\in E\big{(}\mathcal{N}(G^{-})\times \mathcal{N}(H^{-})\big{)}$ if $g\neq g'$, and $(g,h)(g',h')\in E\big{(}\mathcal{N}(G^{-})\square \mathcal{N}(H^{-})\big{)}$ otherwise. Thus, $(g,h)(g',h')$ is an edge of $\mathcal{N}(G^{-})\boxtimes \mathcal{N}(H^{-})$. More precisely, $(g,h)(g',h')$ is an edge of $\big{(}\mathcal{N}(G^{-})\boxtimes \mathcal{N}(H^{-})\big{)}^{-}$.

Conversely, let $(g,h)(g',h')$ be an edge in $\big{(}\mathcal{N}(G^{-})\boxtimes \mathcal{N}(H^{-})\big{)}^{-}$. Without loss of generality, we may assume that $h\neq h'$. It is easily checked that the converse of the above implications hold. So, $(g,h)(g',h')$ is an edge of $\mathcal{N}(G\times H)^{-}$. In fact, we have proved that $E\big{(}\mathcal{N}(G\times H)^{-}\big{)}=E\big{(}\big{(}\mathcal{N}(G^{-})\boxtimes \mathcal{N}(H^{-})\big{)}^{-}\big{)}$. Now the identity mapping $(g,h)\rightarrow(g,h)$ serves as an isomorphism between $\mathcal{N}(G\times H)^{-}$ and $\big{(}\mathcal{N}(G^{-})\boxtimes \mathcal{N}(H^{-})\big{)}^{-}$.

On the other hand, it is clear that $\mathcal{N}(G\times H)\cong \mathcal{N}(G\times H)^{-}+\overline{K_{p}}$ in which $p$ is the cardinality of
$$I_{\mathcal{N}(G\times H)}=\{(g,h)\mid \mbox{for each $(g',h')$, $N_{G}(g)\cap N_{G}(g')=\emptyset$ or $N_{H}(h)\cap N_{H}(h')=\emptyset$}\}.$$
In fact, we have proved that
$$\mathcal{N}(G\times H)\cong \big{(}\mathcal{N}(G^{-})\boxtimes \mathcal{N}(H^{-})\big{)}^{-}+\overline{K_{p}},$$
as desired.
\end{proof}

If $\{G_{1},\dots,G_{r}\}$ and $\{H_{1},\dots,H_{s}\}$ are the sets of components of $G$ and $H$, respectively, then it is clear that $\chi_{i}(G\times H)=\sum_{i,j}\chi_{i}(G_{i}\times H_{j})$. So, we may assume that both $G$ and $H$ are connected. Moreover, $G\times H$ will be an empty graph if either $G$ or $H$ is isomorphic to $K_1$. Therefore, it only suffices to assume that $|V(G)|,|V(H)|\geq 2$.

\begin{theorem}\label{direct}
Let $G$ and $H$ be two connected graphs of orders at least two. If $G\cong K_{2}$ or $H\cong K_{2}$, then $\chi_{i}(G\times H)=\max\{\chi_{i}(G),\chi_{i}(H)\}$. If $\Delta(G),\Delta(H)\geq2$, then
$$\max\{\chi_{i}(G)+\Delta(H),\chi_{i}(H)+\Delta(G)\}\leq \chi_{i}(G\times H)\leq \chi_{i}(G)\chi_{i}(H).$$
These bounds are sharp.
\end{theorem}

\begin{proof}
By using \eqref{open} and Lemma \ref{Iso}, we get
$$\chi_{i}(G\times H)=\chi\big{(}\mathcal{N}(G\times H)\big{)}=\chi\Big{(}\big{(}\mathcal{N}(G^{-})\boxtimes \mathcal{N}(H^{-})\big{)}^{-}+\overline{K_{p}}\Big{)}=\chi\big{(}\mathcal{N}(G^{-})\boxtimes \mathcal{N}(H^{-})\big{)}.$$
Thus, we deduce that $\chi_{i}(G\times H)\le \chi\big{(}\mathcal{N}(G^{-})\big{)}\chi\big{(}\mathcal{N}(H^{-})\big{)}=\chi_{i}(G^{-})\chi_{i}(H^{-})=\chi_{i}(G)\chi_{i}(H)$ since the chromatic number of the strong product of two graphs is always at most the product of chromatic numbers of the factors (see \cite{Klav}). A trivial example that shows the tightness of this upper bound is the direct product graph $K_r\times K_t$ for $r,t\geq3$.

On the other hand, $\chi_{i}(G\times H)=\chi\big{(}\mathcal{N}(G^{-})\boxtimes \mathcal{N}(H^{-})\big{)}\ge \max\{ \chi\big{(}\mathcal{N}(G^{-})\big{)},\chi\big{(}\mathcal{N}(H^{-})\big{)}\}=\max\{\chi_{i}(G^{-}),\chi_{i}(H^{-})\}=\max\{\chi_{i}(G),\chi_{i}(H)\}$.

If $G\cong K_2$, then $\mathcal{N}(G)$ is an empty graph and $\chi_{i}(G)=1$. Thus,
$$\chi_{i}(G\times H)=\max\{\chi_{i}(G),\chi_{i}(H)\}=\chi_{i}(G)\chi_{i}(H)=\chi_{i}(H).$$
A similar situation happens if $H\cong K_2$.

Now, let $\Delta(G),\Delta(H)\geq2$. Then, $G$ has $\Delta(G)$ vertices having a common neighbor. Thus, $\omega\big{(}\mathcal{N}(G)\big{)}\geq \Delta(G)$, in which $\omega$ denotes the clique number. This leads to
$$\chi\big{(}\mathcal{N}(G^{-})\boxtimes \mathcal{N}(H^{-})\big{)}\ge \chi\big{(}K_{\Delta(G)}\boxtimes \mathcal{N}(H^{-})\big{)}\ge \Delta(G)+\chi\big{(}\mathcal{N}(H^{-})\big{)}.$$
The second inequality follows from \cite{Jha} (see also \cite{Klav}). Consequently, by \eqref{open} and Lemma \ref{Iso}, we deduce that $\chi_{i}(G\times H)\ge \chi_{i}(H)+\Delta(G)$. Analogously, we get $\chi_{i}(G\times H)\ge \chi_{i}(G)+\Delta(H)$, and therefore $\chi_{i}(G\times H)\geq \max\{\chi_{i}(G)+\Delta(H),\chi_{i}(H)+\Delta(G)\}$.

To see the sharpness of the bound, we consider the direct product $C_{2r+1}\times C_{2t+1}$ with $r,t\ge 2$. From \cite{Klav}, it is known that $\chi(C_{2r+1}\boxtimes C_{2t+1})=5$. Thus, we have $\chi_{i}(C_{2r+1}\times C_{2t+1})=\chi\big{(}\mathcal{N}(C_{2r+1}\times C_{2t+1})\big{)}=\chi\big{(}\mathcal{N}(C_{2r+1})\boxtimes \mathcal{N}(C_{2t+1})\big{)}=\chi(C_{2r+1}\boxtimes C_{2t+1})=5=\max\{\chi_{i}(C_{2r+1}),\chi_{i}(C_{2t+1})\}+2$, where the third equality comes from the fact that the two-step graph of an odd cycle is isomorphic to itself.
\end{proof}


\subsection{Direct product of two cycles}

The $2$-distance chromatic number of the Cartesian product of two cycles (namely, the torus graphs) has been widely investigated in several papers (for example, see \cite{chmm}, \cite{sv} and \cite{sw}). In $2015$, Chegini et al. (\cite{chmm}) proved that $\chi_{2}(C_{m}\square C_{n})\leq6$ for all integers $m,n\geq10$. Also, the injective chromatic number of torus graphs has recently been studied in \cite{yl}.

Regarding the direct product of two cycles, Kim et al. \cite{ksr} proved the following result.

\begin{theorem}\label{Kim}
The following statements hold.\vspace{1mm}\\
$(i)$ If $m\geq40$ and $n\geq48$ are even, then
\begin{equation*}
\chi_{2}(C_{m}\times C_{n})=\left \{
\begin{array}{lll}
5 & \mbox{if $m,n\equiv0$ $($mod $5$$)$,}\\
6 & \mbox{otherwise.}\\
\end{array}
\right.
\end{equation*}
$(ii)$ If $m\geq40$ is even and $n\geq25$ is odd, then
\begin{equation*}
\chi_{2}(C_{m}\times C_{n})=\left \{
\begin{array}{lll}
5 & \mbox{if $m,n\equiv0$ $($mod $5$$)$,}\\
6 & \mbox{otherwise.}\\
\end{array}
\right.
\end{equation*}
$(iii)$ If $m\geq45$ and $n\geq53$ are odd, then
\begin{equation*}
\chi_{2}(C_{m}\times C_{n})=\left \{
\begin{array}{lll}
5 & \mbox{if $m,n\equiv0$ $($mod $5$$)$,}\\
6 & \mbox{otherwise.}\\
\end{array}
\right.
\end{equation*}
\end{theorem}
In concordance with this, we see that the exact values of the $2$-distance chromatic number of $C_{m}\times C_{n}$ is not yet known for a large number of values of $m,n\geq3$. In contrast with this, in our investigation we make a related and complete study of the injective chromatic number of these graphs. First let
$$S(r,s)=\{\alpha r+\beta s\mid \mbox{$\alpha$ and $\beta$ are nonnegative integers}\}$$
for any two integers $r$ and $s$. We make use of the following useful lemma from number theory due to Sylvester.
\begin{lemma}\emph{(Sylvester \cite{jj})}\label{jjs}
Let $r$ and $s$ be relatively prime integers greater than one. Then $t\in S(r,s)$ for each $t\geq(r-1)(s-1)$, and $(r-1)(s-1)-1\notin S(r,s)$.
\end{lemma}

Regarding injective coloring, the isomorphism
$$\mathcal{N}(G\times H)\cong \big{(}\mathcal{N}(G^{-})\boxtimes \mathcal{N}(H^{-})\big{)}^{-}+\overline{K_{p}},$$
with $p=|I_{\mathcal{N}(G\times H)}|$ (given in Lemma \ref{Iso}), provides us with a useful tool so as to obtain the exact values of $\chi_{i}(C_{m}\times C_{n})$ for all possible values of $m$ and $n$. Along with the above isomorphism, we shall need the following result due to Vesztergombi \cite{Veszter} in $1979$ (see also \cite{Klav}).

\begin{lemma}\emph{(\cite{Klav,Veszter})}\label{Vesz}
For any $s,t\geq2$, $\chi(C_{2s+1}\boxtimes C_{2t+1})=5$.
\end{lemma}

We are now in a position to present the main result of this section. Note that, in the following theorem, $\chi_{i}(C_{m}\times C_{n})$ for the other possible values for $m$ and $n$ which are not appeared here can be determined by taking into account the trivial isomorphism $C_{m}\times C_{n}\cong C_{n}\times C_{m}$ for all $m,n\geq3$.

\begin{theorem}\label{Direct-Cycles}
For any integers $m,n\geq3$,
$$\chi_{i}(C_{m}\times C_{n})=\left \{
\begin{array}{lll}
4 & \ \ \mbox{if $m,n\equiv0$ \emph{(}mod $4$\emph{)}},\\
5 & \ \ \mbox{if ``$m\neq6$ is even and $n\geq5$ is odd" or ``both $m,n\geq5$ are odd"}\\
  &  \ \ \mbox{or ``$(m,n)=(4s+2,4t+2)$ with $s,t\geq2$"}\\
  &   \ \ \mbox{or ``$(m,n)=(4,4t+2)$ with $t\geq2$"},\\
6 & \ \ \mbox{if $m=4s$ for $s\geq1$ and $t\in \{3,6\}$},\\
7 & \ \ \mbox{if $m\in \{3,6\}$ and $n=2t+1$ with $t\geq3$},\\
8 & \ \ \mbox{if $m\in \{3,6\}$ and $n=5$},\\
9 & \ \ \mbox{if $m\in \{3,6\}$ and $n=3$}.\\
\end{array}
\right.$$
\end{theorem}
\begin{proof}
We distinguish three cases depending on the parity of $m$ and $n$ by taking into account the facts that $C_{m}\times C_{n}\cong C_{n}\times C_{m}$ and $C_{m}\boxtimes C_{n}\cong C_{n}\boxtimes C_{m}$ for each $m,n\geq3$.\vspace{1mm}\\
\textit{Case 1.} Both $m$ and $n$ are odd. Let $m=2s+1$ and $n=2t+1$ for some $s,t\geq1$. It can be easily observed that if $G$ is a cycle of order $n\geq 3$, then $\mathcal{N}(G)$ is also a cycle of order $n$ when $n$ is odd. Suppose first that $s,t\geq2$. We then deduce from (\ref{open-previous}), Lemma \ref{Iso} and Lemma \ref{Vesz} that
$$\chi_{i}(C_{2s+1}\times C_{2t+1})=\chi(C_{2s+1}\boxtimes C_{2t+1})=5.$$

We now assume by symmetry that one of the factors, say $C_{2s+1}$, is of order three. From \cite{Klav} we know that $\chi(K_{m}\boxtimes C_{2n+1})=2m+\lceil m/n\rceil$ for $m\geq1$ and $n\geq2$. This shows that $\chi_{i}(C_{3}\times C_{5})=8$ and that $\chi_{i}(C_{3}\times C_{2t+1})=7$ for $t\geq3$. On the other hand, it is readily seen that $\chi_{i}(C_{3}\times C_{3})=9$.\vspace{1mm}\\
\textit{Case 2.} Suppose that $m$ is even and $n=2t+1$ for some $t\geq1$. Suppose first that $m=4k+2$ for some $k\geq1$. Notice that if $G$ is a cycle of order $p$, then $\mathcal{N}(G)$ is the disjoint union of two cycles of order $p/2$ when $p\geq6$ is even, and it is $K_{2}+K_{2}$ if $p=4$. With this in mind, we conclude that
$$\chi_{i}(C_{m}\times C_{2t+1})=\chi\big{(}(C_{2k+1}+C_{2k+1})\boxtimes C_{2t+1}\big{)}=\chi(C_{2k+1}\boxtimes C_{2t+1}).$$
Therefore, $\chi_{i}(C_{4k+2}\times C_{2t+1})=5$ for all $k,t\geq2$ (by using Lemma \ref{Vesz}). So, we need to discuss the cases when $k=1$ and when $t=1$ separately. In particular, we have $\chi_{i}(C_{6}\times C_{5})=8$ and $\chi_{i}(C_{6}\times C_{2t+1})=7$ for $t\geq3$. Moreover, it is easy to see that $\chi_{i}(C_{6}\times C_{3})=\chi(C_{3}\boxtimes C_{3})=9$. Also, the possible values for $\chi_{i}(C_{4k+2}\times C_{3})=\chi(C_{2k+1}\boxtimes C_{3})$ has just been discussed.

We now assume that $m=4k$ for an integer $k\geq1$. In what follows, we take advantage of the following useful claim.\vspace{1mm}\\
Claim $1$. {\em Let $k\geq2$ be an integer. For any odd integer $n\geq3$,}
$$\chi(C_{2k}\boxtimes C_{n})=\left \{
\begin{array}{lll}
5 & \mbox{if $n\geq5$},\\
6 & \mbox{if $n=3$}.
\end{array}
\right.$$
\textit{Proof of Claim 1.}
We observe that $\omega(C_{3}\boxtimes C_{2k})=6$, in which $\omega$ denotes the clique number. Therefore, the pattern given in Figure \ref{Pat11} represents an optimal $6$-coloring of $C_{3}\boxtimes C_{2k}\cong C_{2k}\boxtimes C_{3}$ for each $k\geq2$. So, from now on, we assume that $n\geq5$ (odd).
\begin{figure}[ht]
 \centering
\begin{tikzpicture}[scale=.02, transform shape]
\node [draw, shape=circle] (v1) at (0,0) {};
\node [draw, shape=circle] (v2) at (120,0) {};
\node [draw, shape=circle] (v3) at (0,-60) {};
\node [draw, shape=circle] (v4) at (120,-60) {};

\node [scale=50] at (10,-10) {1};
\node [scale=50] at (10,-30) {2};
\node [scale=50] at (10,-50) {3};

\node [scale=50] at (30,-10) {4};
\node [scale=50] at (30,-30) {5};
\node [scale=50] at (30,-50) {6};

\node [scale=50] at (60,-10) {$\cdots$};
\node [scale=50] at (60,-30) {$\cdots$};
\node [scale=50] at (60,-50) {$\cdots$};

\node [scale=50] at (90,-10) {1};
\node [scale=50] at (90,-30) {2};
\node [scale=50] at (90,-50) {3};

\node [scale=50] at (110,-10) {4};
\node [scale=50] at (110,-30) {5};
\node [scale=50] at (110,-50) {6};

\draw (v1)--(v2)--(v4)--(v3)--(v1);

\end{tikzpicture}
  \caption{An optimal $6$-coloring of $C_{3}\boxtimes C_{2k}$ for each $k\geq2$.}\label{Pat11}
\end{figure}

It is shown in \cite{Vesel} that $\alpha(C_{2i}\boxtimes C_{2j+1})=ij$, for all positive integers $i\geq2$ and $j$, in which $\alpha$ stands for the independence number. Therefore,
\begin{equation}\label{five}
\chi(C_{2k}\boxtimes C_{2t+1})\geq \lceil(2k)(2t+1)/kt\rceil=5.
\end{equation}

It is readily checked that the patterns $A$ and $B$ depicted in Figure \ref{Pat22} represent $5$-colorings of $C_{4}\boxtimes C_{5}$ and $C_{4}\boxtimes C_{4}$, respectively. Moreover, the combined patterns $AB$ and $BA$ provide $5$-colorings of $C_{4}\boxtimes C_{9}$.
\begin{figure}[ht]
 \centering
\begin{tikzpicture}[scale=.02, transform shape]
\node [draw, shape=circle] (v1) at (-300,0) {};
\node [draw, shape=circle] (v2) at (-200,0) {};
\node [draw, shape=circle] (v3) at (-300,-80) {};
\node [draw, shape=circle] (v4) at (-200,-80) {};

\node [scale=50] at (-290,-10) {1};
\node [scale=50] at (-290,-30) {3};
\node [scale=50] at (-290,-50) {5};
\node [scale=50] at (-290,-70) {3};

\node [scale=50] at (-270,-10) {2};
\node [scale=50] at (-270,-30) {4};
\node [scale=50] at (-270,-50) {1};
\node [scale=50] at (-270,-70) {4};

\node [scale=50] at (-250,-10) {3};
\node [scale=50] at (-250,-30) {5};
\node [scale=50] at (-250,-50) {2};
\node [scale=50] at (-250,-70) {5};

\node [scale=50] at (-230,-10) {4};
\node [scale=50] at (-230,-30) {1};
\node [scale=50] at (-230,-50) {3};
\node [scale=50] at (-230,-70) {1};

\node [scale=50] at (-210,-10) {5};
\node [scale=50] at (-210,-30) {2};
\node [scale=50] at (-210,-50) {4};
\node [scale=50] at (-210,-70) {2};

\draw (v1)--(v2)--(v4)--(v3)--(v1);
\node [scale=50] at (-315,-40) {$A$:};


\node [draw, shape=circle] (u1) at (-140,0) {};
\node [draw, shape=circle] (u2) at (-60,0) {};
\node [draw, shape=circle] (u3) at (-140,-80) {};
\node [draw, shape=circle] (u4) at (-60,-80) {};

\node [scale=50] at (-130,-10) {4};
\node [scale=50] at (-130,-30) {1};
\node [scale=50] at (-130,-50) {3};
\node [scale=50] at (-130,-70) {1};

\node [scale=50] at (-110,-10) {5};
\node [scale=50] at (-110,-30) {2};
\node [scale=50] at (-110,-50) {4};
\node [scale=50] at (-110,-70) {2};

\node [scale=50] at (-90,-10) {4};
\node [scale=50] at (-90,-30) {1};
\node [scale=50] at (-90,-50) {3};
\node [scale=50] at (-90,-70) {1};

\node [scale=50] at (-70,-10) {5};
\node [scale=50] at (-70,-30) {2};
\node [scale=50] at (-70,-50) {4};
\node [scale=50] at (-70,-70) {2};

\draw (u1)--(u2)--(u4)--(u3)--(u1);
\node [scale=50] at (-155,-40) {$B$:};


\node [draw, shape=circle] (v1) at (0,0) {};
\node [draw, shape=circle] (v2) at (140,0) {};
\node [draw, shape=circle] (v3) at (0,-80) {};
\node [draw, shape=circle] (v4) at (140,-80) {};

\node [scale=50] at (10,-10) {1};
\node [scale=50] at (10,-30) {3};
\node [scale=50] at (10,-50) {5};
\node [scale=50] at (10,-70) {3};

\node [scale=50] at (30,-10) {2};
\node [scale=50] at (30,-30) {4};
\node [scale=50] at (30,-50) {1};
\node [scale=50] at (30,-70) {4};

\node [scale=50] at (50,-10) {3};
\node [scale=50] at (50,-30) {5};
\node [scale=50] at (50,-50) {2};
\node [scale=50] at (50,-70) {5};

\node [scale=50] at (70,-10) {4};
\node [scale=50] at (70,-30) {1};
\node [scale=50] at (70,-50) {3};
\node [scale=50] at (70,-70) {1};

\node [scale=50] at (90,-10) {5};
\node [scale=50] at (90,-30) {2};
\node [scale=50] at (90,-50) {4};
\node [scale=50] at (90,-70) {2};

\node [scale=50] at (110,-10) {4};
\node [scale=50] at (110,-30) {1};
\node [scale=50] at (110,-50) {3};
\node [scale=50] at (110,-70) {1};

\node [scale=50] at (130,-10) {5};
\node [scale=50] at (130,-30) {2};
\node [scale=50] at (130,-50) {4};
\node [scale=50] at (130,-70) {2};

\draw (v1)--(v2)--(v4)--(v3)--(v1);
\node [scale=50] at (-15,-40) {$C$:};


\node [draw, shape=circle] (u1) at (200,0) {};
\node [draw, shape=circle] (u2) at (420,0) {};
\node [draw, shape=circle] (u3) at (200,-80) {};
\node [draw, shape=circle] (u4) at (420,-80) {};

\node [scale=50] at (210,-10) {1};
\node [scale=50] at (210,-30) {3};
\node [scale=50] at (210,-50) {5};
\node [scale=50] at (210,-70) {3};

\node [scale=50] at (230,-10) {2};
\node [scale=50] at (230,-30) {4};
\node [scale=50] at (230,-50) {1};
\node [scale=50] at (230,-70) {4};

\node [scale=50] at (250,-10) {3};
\node [scale=50] at (250,-30) {5};
\node [scale=50] at (250,-50) {2};
\node [scale=50] at (250,-70) {5};

\node [scale=50] at (270,-10) {4};
\node [scale=50] at (270,-30) {1};
\node [scale=50] at (270,-50) {3};
\node [scale=50] at (270,-70) {1};

\node [scale=50] at (290,-10) {5};
\node [scale=50] at (290,-30) {2};
\node [scale=50] at (290,-50) {4};
\node [scale=50] at (290,-70) {2};

\node [scale=50] at (310,-10) {4};
\node [scale=50] at (310,-30) {1};
\node [scale=50] at (310,-50) {3};
\node [scale=50] at (310,-70) {1};

\node [scale=50] at (330,-10) {5};
\node [scale=50] at (330,-30) {2};
\node [scale=50] at (330,-50) {4};
\node [scale=50] at (330,-70) {2};

\node [scale=50] at (350,-10) {4};
\node [scale=50] at (350,-30) {1};
\node [scale=50] at (350,-50) {3};
\node [scale=50] at (350,-70) {1};

\node [scale=50] at (370,-10) {5};
\node [scale=50] at (370,-30) {2};
\node [scale=50] at (370,-50) {4};
\node [scale=50] at (370,-70) {2};

\node [scale=50] at (390,-10) {4};
\node [scale=50] at (390,-30) {1};
\node [scale=50] at (390,-50) {3};
\node [scale=50] at (390,-70) {1};

\node [scale=50] at (410,-10) {5};
\node [scale=50] at (410,-30) {2};
\node [scale=50] at (410,-50) {4};
\node [scale=50] at (410,-70) {2};

\draw (u1)--(u2)--(u4)--(u3)--(u1);
\node [scale=50] at (185,-40) {$D$:};

\end{tikzpicture}
  \caption{The patterns $A$, $B$, $C$ and $D$.}\label{Pat22}
\end{figure}
By Lemma \ref{jjs} and using combinations of the patterns $A$ and $B$, we obtain a $(4\times n)$-pattern $C(n)$, as a $5$-coloring of $C_{4}\boxtimes C_{n}$ for each odd integer $n\geq13$. Using a combination of $k\geq1$ copies of $C(n)$, we get a $(4k\times n)$-pattern as a $5$-coloring of $C_{4k}\boxtimes C_{n}$ for every odd integer $n\geq13$.

In addition, we observe that the patterns $C$ and $D$ given in Figure \ref{Pat22} represent $5$-colorings of $C_{4}\boxtimes C_{7}$ and $C_{4}\boxtimes C_{11}$, respectively. Analogously, a combination of $k$ copies of $A,C,AB,D$ give us a $5$-coloring of $C_{4k}\boxtimes C_{j}$ for $j=5,7,9,11$, respectively. In fact, with the inequality (\ref{five}) in mind, we have proved that $\chi(C_{4k}\boxtimes C_{n})=5$ for all $k\geq1$ and odd integer $n\geq5$.

Consider the direct product graph $C_{4t+2}\boxtimes C_{n}$ for any $t\geq1$ and odd integer $n\geq5$. Let $H$ be a $(4t\times n)$-pattern as a $5$-coloring of $C_{4t}\boxtimes C_{n}$ as above. Let $H'$ be a subpattern of $H$ consisting of the first two rows of $H$. It is easy to check that the $\big{(}(4t+2)\times n\big{)}$-pattern $H''$ obtained from $H$ by considering $H'$ as its last two rows provides us with a $5$-coloring of $C_{4t+2}\boxtimes C_{n}$ for any $t\geq1$ and odd integer $n\geq5$. Consequently, we have $\chi(C_{4t+2}\boxtimes C_{n})=5$ for all $t\geq1$ and odd integer $n\geq5$. The above argument shows that $\chi(C_{2k}\boxtimes C_{n})=5$ for each $k\geq2$ and odd integer $n\geq5$. $(\square)$\vspace{1mm}

We now infer from Claim $1$ that $\chi_{i}(C_{4k}\times C_{n})=\chi(C_{2k}\boxtimes C_{n})=5$ for each $k\geq2$ and odd integer $n\geq5$, and that $\chi_{i}(C_{4k}\times C_{3})=6$ for each $k\geq2$. Furthermore, in the case when $k=1$, we have $\chi_{i}(C_{4}\times C_{2t+1})=\chi(K_{2}\boxtimes C_{2t+1})$. Therefore, $\chi_{i}(C_{4}\times C_{2t+1})=5$ for $t\geq2$. Moreover, it is easy to see that $\chi_{i}(C_{4}\times C_{3})=\chi(K_{2}\boxtimes C_{3})=6$.\vspace{1mm}\\
\textit{Case 3.} Both $m$ and $n$ are even. If $m=4s+2$ and $n=4t+2$ for some $s,t\geq1$, then $\chi_{i}(C_{4s+2}\times C_{4t+2})=\chi(C_{2s+1}\boxtimes C_{2t+1})$. Therefore, $\chi_{i}(C_{4s+2}\times C_{4t+2})=5$ if $s,t\geq2$ as we discussed in Case $1$. On the other hand, for the remaining possible values of $s$ and $t$, it suffices to consider the case when $s=1$. In such a situation, we have $\chi_{i}(C_{6}\times C_{2t+1})=\chi(C_{3}\boxtimes C_{2t+1})$. Hence, $\chi_{i}(C_{6}\times C_{2t+1})=7$ for $t\geq3$. Moreover, it is easy to see that $\chi_{i}(C_{6}\times C_{3})=\chi(C_{3}\boxtimes C_{3})=9$ and $\chi_{i}(C_{6}\times C_{5})=\chi(C_{3}\boxtimes C_{5})=8$.

If $m=4s$ and $n=4t+2$ for some $s,t\geq1$, then $\chi_{i}(C_{4s}\times C_{4t+2})=\chi(C_{2s}\boxtimes C_{2t+1})$ for $s\geq2$ and $t\geq1$. In such a situation, Claim 1 implies that $\chi_{i}(C_{4s}\times C_{6})=6$ for $s\geq2$, and $\chi_{i}(C_{4s}\times C_{4t+2})=5$ for $s,t\geq2$. We also have $\chi_{i}(C_{4}\times C_{6})=6$ and $\chi_{i}(C_{4}\times C_{4t+2})=\chi(K_{2}\boxtimes C_{2t+1})=5$ for $t\geq2$.

Finally, let $m=4s$ and $n=4t$ for some $s,t\geq1$. It is clear that the $(2s\times2t)$-pattern depicted in Figure \ref{Pat44} gives us a $4$-coloring of $C_{2s}\boxtimes C_{2t}$ for each $s,t\geq1$.
\begin{figure}[ht]
 \centering
\begin{tikzpicture}[scale=.02, transform shape]
\node [draw, shape=circle] (v1) at (0,0) {};
\node [draw, shape=circle] (v2) at (160,0) {};
\node [draw, shape=circle] (v3) at (0,-150) {};
\node [draw, shape=circle] (v4) at (160,-150) {};

\node [scale=50] at (10,-10) {1};
\node [scale=50] at (30,-10) {2};
\node [scale=50] at (10,-30) {3};
\node [scale=50] at (30,-30) {4};

\node [scale=50] at (50,-10) {1};
\node [scale=50] at (70,-10) {2};
\node [scale=50] at (50,-30) {3};
\node [scale=50] at (70,-30) {4};

\node [scale=50] at (100,-10) {$\cdots$};
\node [scale=50] at (100,-30) {$\cdots$};

\node [scale=50] at (130,-10) {1};
\node [scale=50] at (150,-10) {2};
\node [scale=50] at (130,-30) {3};
\node [scale=50] at (150,-30) {4};


\node [scale=50] at (10,-50) {1};
\node [scale=50] at (30,-50) {2};
\node [scale=50] at (10,-70) {3};
\node [scale=50] at (30,-70) {4};

\node [scale=50] at (50,-50) {1};
\node [scale=50] at (70,-50) {2};
\node [scale=50] at (50,-70) {3};
\node [scale=50] at (70,-70) {4};

\node [scale=50] at (100,-50) {$\cdots$};
\node [scale=50] at (100,-70) {$\cdots$};

\node [scale=50] at (130,-50) {1};
\node [scale=50] at (150,-50) {2};
\node [scale=50] at (130,-70) {3};
\node [scale=50] at (150,-70) {4};

\node [scale=50] at (20,-90) {$\vdots$};
\node [scale=50] at (60,-90) {$\vdots$};
\node [scale=50] at (140,-90) {$\vdots$};


\node [scale=50] at (10,-120) {1};
\node [scale=50] at (30,-120) {2};
\node [scale=50] at (10,-140) {3};
\node [scale=50] at (30,-140) {4};

\node [scale=50] at (50,-120) {1};
\node [scale=50] at (70,-120) {2};
\node [scale=50] at (50,-140) {3};
\node [scale=50] at (70,-140) {4};

\node [scale=50] at (100,-120) {$\cdots$};
\node [scale=50] at (100,-140) {$\cdots$};

\node [scale=50] at (130,-120) {1};
\node [scale=50] at (150,-120) {2};
\node [scale=50] at (130,-140) {3};
\node [scale=50] at (150,-140) {4};

\draw (v1)--(v2)--(v4)--(v3)--(v1);

\end{tikzpicture}
  \caption{An optimal $4$-coloring of $C_{2s}\boxtimes C_{2t}$ for each $s,t\geq1$. Here we let $C_{2}=K_{2}$ for the sake of convenience.}\label{Pat44}
\end{figure}
We here assume $C_{2}=K_{2}$ for the sake of convenience. On the other hand, we get $\chi_{i}(C_{4s}\times C_{4t+2})=\chi(C_{2s}\boxtimes C_{2t})=4$ since $\omega(C_{2s}\boxtimes C_{2t})=4$. This completes the proof.
\end{proof}


\section{Lexicographic and corona product graphs}\label{Sect:pro}

Our first aim in this section is to give sharp lower and upper bounds on $\chi_{i}(G\circ H)$. We also prove that $\chi_{2}$ and $\chi_{i}$ are the same in the case of lexicographic product graphs when both $G$ and $H$ have no isolated vertices.

\begin{theorem}\label{Lexico}
Let $G$ be a connected graph of order at least two and let $H$ be any graph with $i_{H}$ isolated vertices. Then,
$$\chi_{i}(G\circ H)\leq \chi_{2}(G)|V(H)|-i_{H}\big{(}\chi_{2}(G)-\chi_{i}(G)\big{)}.$$
Moreover, if $H$ has no isolated vertices, then
$$\chi_{i}(G\circ H)=\chi_{2}(G\circ H)\geq(\Delta(G)+1)|V(H)|.$$
\end{theorem}
\begin{proof}
Let $\mathbb{A}=\{A_{1},\dots,A_{\chi_{i}(G)}\}$ and $\mathbb{B}=\{B_{1},\dots,B_{\chi_{2}(G)}\}$ be a $\chi_{i}(G)$-coloring and a $\chi_{2}(G)$-coloring, respectively. Also, let $I_{H}$ be the set of isolated vertices of $H$. We set
$$\mathbb{P}=\{A_{i}\times \{h\}\mid 1\leq i\leq \chi_{i}(G), h\in I_{H}\}\cup \{B_{i}\times \{h\}\mid 1\leq i\leq \chi_{2}(G), h\in V(H)\setminus I_{H}\}.$$
Clearly, $\mathbb{P}$ is a vertex partition of $G\circ H$.

Suppose that there exists a vertex $(g,h')$ adjacent to two distinct vertices $(g',h),(g'',h)\in A_{i}\times \{h\}$, for some $1\leq i\leq \chi_{i}(G)$ and $h\in I_{H}$. Note first that $g$ cannot simultaneously be adjacent to both $g'$ and $g''$, due to the fact that $A_{i}$ is an open packing in $G$. So, it must happen, without loss of generality, that $g=g''$. However, this means that $hh'$ is an edge of $H$, which is a contradiction to the fact that $h$ is an isolated vertex of $H$. Therefore, $A_{i}\times \{h\}$ is an open packing in $G\circ H$.

If $(g,h')$ is adjacent to two distinct vertices $(g',h),(g'',h)\in B_{i}\times \{h\}$ for some $1\leq i\leq \chi_{2}(G)$ and $h\in V(H)\setminus I_{H}$, then $g',g''\in N_{G}[g]\cap B_{i}$. This is a contradiction since any two vertices of $B_{i}$ are at distance larger than two in $G$. This shows that $B_{i}\times \{h\}$ is a open packing in $G\circ H$. So, we have concluded that $\mathbb{P}$ is an injective coloring of $G\circ H$. Therefore,
\begin{equation}\label{Lexico1}
\chi_{i}(G\circ H)\leq|\mathbb{P}|=\chi_{i}(G)i_{H}+\chi_{2}(G)\big{(}|V(H)|-i_{H}\big{)}=\chi_{2}(G)|V(H)|-i_{H}\big{(}\chi_{2}(G)-\chi_{i}(G)\big{)}.
\end{equation}

The bound is sharp for a large number of infinite families of graphs. For instance, consider the lexicographic product graph $F=P_{r}\circ(P_{s}+\overline{P_{t}})$ with $r\geq3$, $s\geq2$ and $t\geq1$. Let $V(P_r)=\{u_{1},\dots,u_{r}\}$, $V(P_s)=\{v_{1},\dots,v_{s}\}$ and $V(\overline{P_t})=\{w_{1},\dots,w_{t}\}$. Set $F'=F[\{u_{1},u_{2},u_{3}\}\times\big{(}V(P_{s})\cup V(\overline{P_{t}})\big{)}]$. Let $f':V(F')\rightarrow\{1,2,\dots,\chi_{i}(F')\}$ be any $\chi_{i}(F')$-coloring. Note that in the subgraph of $F'$ induced by $\{u_{1},u_{2},u_{3}\}\times V(P_{s})$, each edge lies on a triangle. Therefore, no two vertices of this induced subgraph receive the same color by $f'$. This show that $f'$ assign $3s$ colors to the vertices of $\{u_{1},u_{2},u_{3}\}\times V(P_{s})$. On the other hand, because $(u_{2},v_{1})$ is adjacent to all vertices in $\{u_{1},u_{3}\}\times\big{(}V(P_{s})\cup V(P_{t})\big{)}$, it follows that every vertex in this set receives a unique color by $f'$. In fact, we observe that $f'$ assigns $2t$ colors to the vertices in $Q=\{u_{1},u_{3}\}\times V(P_{t})$ and that $f'\big{(}\{u_{1},u_{2},u_{3}\}\times V(P_{s})\big{)}\cap f'(Q)=\emptyset$. The above discussion shows that
\begin{equation*}
\begin{array}{lcl}
\chi_{i}(F)\geq \chi_{i}(F')&=&|f'\big{(}\{u_{1},u_{2},u_{3}\}\times V(P_{s})\big{)}|+\sum_{i=1}^{t}|f'(\{u_{1},u_{2},u_{3}\}\times\{w_{i}\})|\\
&\geq&3s+2t\\
&=&\chi_{2}(P_{r})|V(P_{s}+\overline{P_{t}})|-i_{P_{s}+\overline{P_{t}}}(\chi_{2}(P_{r})-\chi_{i}(P_{r})).
\end{array}
\end{equation*}
This results in the equality in the upper bound.

Suppose now that $H$ has no isolated vertices. Let $(g,h)(g',h')\in E(G\circ H)$. If $g=g'$, then $(g'',h)$ is adjacent to both $(g,h)$ and $(g',h')$ in which $g''$ is any vertex of $G$ adjacent to $g$. Suppose that $gg'\in E(G)$. There exists $h''\in V(H)$ adjacent to $h$ because $H$ does not have isolated vertices. So, $(g,h'')$ is adjacent to both $(g,h)$ and $(g',h')$ by the adjacency rule of the lexicographic product graphs. In fact, every edge of $G\circ H$  lies on a triangle. This shows that every open packing in $G\circ H$ is an independent set. In particular, a subset of $V(G\circ H)$ is a packing if and only if it is an open packing. Therefore, $\chi_{i}(G\circ H)=\chi_{2}(G\circ H)$.

Let $g$ be a vertex of $G$ of maximum degree. It happens that $\diam\big{(}(G\circ H)[N_{G}[g]\times V(H)]\big{)}\leq2$. This in particular implies that any $2$-distance coloring of $G\circ H$ assigns at least $|N_{G}[g]\times V(H)|=(\Delta(G)+1)|V(H)|$ colors to the vertices of $G\circ H$. Consequently, $\chi_{i}(G\circ H)=\chi_{2}(G\circ H)\geq(\Delta(G)+1)|V(H)|$.

That the lower bound is sharp, may be seen as follows. It is known that $\chi_{2}(T)=\Delta(T)+1$ for any tree $T$ (see Theorem 2.4 in \cite{pr} for $k=2$). Let $T$ be a nontrivial tree and let $H$ be any graph with no isolated vertices. Then, $\chi_{i}(T\circ H)=\chi_{2}(T\circ H)=(\Delta(T)+1)|V(H)|$ by considering both lower and upper bounds. This completes the proof.
\end{proof}

Ghazi et al. \cite{grt} exhibited the exact formula $\chi_{2}(G\circ H)=\chi_{2}(G)|V(H)|$ for all connected graphs $G$ and $H$. In what follows, we show that this equality is not true as it stands. In Figure \ref{Counterexample}, we consider the lexicographic product graph $C_{7}\circ C_{5}$ without drawing the edges (for the sake of convenience). Note that the assigned numbers to the vertices represent a $2$DC of $C_{7}\circ C_{5}$ with $18$ colors. So, $\chi_{2}(C_{7}\circ C_{5})\leq18<20=\chi_{2}(C_{7})|V(C_{5})|$.

\begin{figure}[ht]
 \centering
\begin{tikzpicture}[scale=.50, transform shape]
\node [draw, shape=circle] (v1) at  (0,0) {};
\node [draw, shape=circle] (v2) at  (2.5,0) {};
\node [draw, shape=circle] (v3) at  (5,0) {};
\node [draw, shape=circle] (v4) at  (7.5,0) {};
\node [draw, shape=circle] (v5) at  (10,0) {};

\node [draw, shape=circle] (v1) at  (0,-1.5) {};
\node [draw, shape=circle] (v2) at  (2.5,-1.5) {};
\node [draw, shape=circle] (v3) at  (5,-1.5) {};
\node [draw, shape=circle] (v4) at  (7.5,-1.5) {};
\node [draw, shape=circle] (v5) at  (10,-1.5) {};

\node [draw, shape=circle] (v1) at  (0,-3) {};
\node [draw, shape=circle] (v2) at  (2.5,-3) {};
\node [draw, shape=circle] (v3) at  (5,-3) {};
\node [draw, shape=circle] (v4) at  (7.5,-3) {};
\node [draw, shape=circle] (v5) at  (10,-3) {};

\node [draw, shape=circle] (v1) at  (0,-4.5) {};
\node [draw, shape=circle] (v2) at  (2.5,-4.5) {};
\node [draw, shape=circle] (v3) at  (5,-4.5) {};
\node [draw, shape=circle] (v4) at  (7.5,-4.5) {};
\node [draw, shape=circle] (v5) at  (10,-4.5) {};

\node [draw, shape=circle] (v1) at  (0,-6) {};
\node [draw, shape=circle] (v2) at  (2.5,-6) {};
\node [draw, shape=circle] (v3) at  (5,-6) {};
\node [draw, shape=circle] (v4) at  (7.5,-6) {};
\node [draw, shape=circle] (v5) at  (10,-6) {};

\node [draw, shape=circle] (v1) at  (0,-7.5) {};
\node [draw, shape=circle] (v2) at  (2.5,-7.5) {};
\node [draw, shape=circle] (v3) at  (5,-7.5) {};
\node [draw, shape=circle] (v4) at  (7.5,-7.5) {};
\node [draw, shape=circle] (v5) at  (10,-7.5) {};

\node [draw, shape=circle] (v1) at  (0,-9) {};
\node [draw, shape=circle] (v2) at  (2.5,-9) {};
\node [draw, shape=circle] (v3) at  (5,-9) {};
\node [draw, shape=circle] (v4) at  (7.5,-9) {};
\node [draw, shape=circle] (v5) at  (10,-9) {};

\node [scale=1.7] at (-0.5,0) {\large $1$};
\node [scale=1.7] at (2,0) {\large $5$};
\node [scale=1.7] at (4.5,0) {\large $8$};
\node [scale=1.7] at (6.8,0) {\large $12$};
\node [scale=1.7] at (9.3,0) {\large $15$};

\node [scale=1.7] at (-0.5,-1.5) {\large $2$};
\node [scale=1.7] at (2,-1.5) {\large $6$};
\node [scale=1.7] at (4.3,-1.5) {\large $9$};
\node [scale=1.7] at (6.8,-1.5) {\large $13$};
\node [scale=1.7] at (9.3,-1.5) {\large $16$};

\node [scale=1.7] at (-0.5,-3) {\large $3$};
\node [scale=1.7] at (2,-3) {\large $7$};
\node [scale=1.7] at (4.3,-3) {\large $10$};
\node [scale=1.7] at (6.8,-3) {\large $14$};
\node [scale=1.7] at (9.3,-3) {\large $17$};

\node [scale=1.7] at (-0.5,-4.5) {\large $1$};
\node [scale=1.7] at (2,-4.5) {\large $4$};
\node [scale=1.7] at (4.5,-4.5) {\large $8$};
\node [scale=1.7] at (6.8,-4.5) {\large $11$};
\node [scale=1.7] at (9.3,-4.5) {\large $15$};

\node [scale=1.7] at (-0.5,-6) {\large $2$};
\node [scale=1.7] at (2,-6) {\large $5$};
\node [scale=1.7] at (4.3,-6) {\large $9$};
\node [scale=1.7] at (6.8,-6) {\large $12$};
\node [scale=1.7] at (9.3,-6) {\large $16$};

\node [scale=1.7] at (-0.5,-7.5) {\large $3$};
\node [scale=1.7] at (2,-7.5) {\large $6$};
\node [scale=1.7] at (4.3,-7.5) {\large $10$};
\node [scale=1.7] at (6.8,-7.5) {\large $13$};
\node [scale=1.7] at (9.3,-7.5) {\large $17$};

\node [scale=1.7] at (-0.5,-9) {\large $4$};
\node [scale=1.7] at (2,-9) {\large $7$};
\node [scale=1.7] at (4.3,-9) {\large $11$};
\node [scale=1.7] at (6.8,-9) {\large $14$};
\node [scale=1.7] at (9.3,-9) {\large $18$};
\end{tikzpicture}
  \caption{A counterexample to the formula $\chi_{2}(G\circ H)=\chi_{2}(G)|V(H)|$ for all connected graphs $G$ and $H$.}\label{Counterexample}
\end{figure}

Let $G$ and $H$ be graphs and $V(G)=\{v_1,\ldots,v_{n}\}$. We recall that the {\em corona product} $G\odot H$ of graphs $G$ and $H$ is obtained from the disjoint union of $G$ and $n$ disjoint copies of $H$, say $H_1,\ldots, H_{n}$, such that the vertex $v_i\in V(G)$ is adjacent to every vertex of $H_i$ for all $i\in \{1,\dots,n\}$. We next present a closed formula for $\chi_{i}(G\odot H)$.

\begin{theorem}\label{Cor}
For any graphs $G$ and $H$ with no isolated vertices,
$$\chi_{i}(G\odot H)\in\big{\{}\chi_{i}(G),|V(H)|+\Delta(G),|V(H)|+\Delta(G)+1\big{\}}.$$
\end{theorem}
\begin{proof}
Clearly, any $\chi_{i}(G\odot H)$-coloring assigns at least $\chi_{i}(G)$ colors to the vertices of $G$ since $G$ is a subgraph of $G\odot H$. So, $\chi_{i}(G\odot H)\geq \chi_{i}(G)$.

For each $1\leq i\leq|V(G)|$, assume that $V(H_{i})=\{u_{i1},\dots,u_{i|V(H)|}\}$. Let $\mathbb{A}=\{A_{1},\dots,A_{\chi_{i}(G)}\}$ be a $\chi_{i}(G)$-coloring. In what follows, we construct a mapping $f$ on $V(G\odot H)$ that assigns the colors $1,\dots,\chi_{i}(G)$ to the vertices in $A_{1},\dots,A_{\chi_{i}(G)}$, respectively. In particular, $f$ turns out to be a $\chi_{i}(G)$-coloring. We consider two cases depending on $\chi_{i}(G)$.\vspace{1.5mm}\\
\textit{Case 1.} $|V(H)|\leq \chi_{i}(G)-\Delta(G)-1$. We choose an arbitrary vertex $v_{i}$ and let it be in $A_k$. We now extend $f$ by assigning $|V(H)|$ colors $r_1,\dots,r_{|V(H)|}\in\{1,\dots,\chi_{i}(G)\}$ to the vertices $u_{i1},\dots,u_{i|V(H)|}$ such that $N_{G}[v_{i}]\cap A_{r_j}=\emptyset$ for every $1\le j\le |V(H)|$ (there do exist such colors since $|V(H)|\leq \chi_{i}(G)-\deg_{G}(v_i)-1$). By iterating this process for all vertices in $V(G)$, we note that $f$ is an injective coloring of $G\odot H$ assigning $\chi_{i}(G)$ colors to the vertices of $G\odot H$. Therefore $\chi_{i}(G\odot H)\leq \chi_{i}(G)$, and hence $\chi_{i}(G\odot H)=\chi_{i}(G)$.\vspace{1.5mm}\\
\textit{Case 2.} $|V(H)|\geq \chi_{i}(G)-\Delta(G)$. We need to distinguish two more possibilities depending also on the behavior of vertices of maximum degree in $G$.\vspace{0.5mm}\\
\textit{Subcase 2.1.} Suppose that we have ``$|V(H)|=\chi_{i}(G)-\Delta(G)$" and the property that ``every vertex $v_j$ of maximum degree in $G$ has a (unique) neighbor in the open packing (color class) from $\mathbb{A}$ containing $v_j$". In such a situation, similarly to the argument given in Case $1$, $G\odot H$ can be injectively colored with $\chi_{i}(G)$ colors. Thus, $\chi_{i}(G\odot H)=\chi_{i}(G)$.\vspace{0.5mm}\\
\textit{Subcase 2.2.} Suppose that ``$|V(H)|>\chi_{i}(G)-\Delta(G)$" or we have ``$|V(H)|=\chi_{i}(G)-\Delta(G)$ with the property that there exists a vertex $v_j$ of maximum degree in $G$ having no neighbor in the open packing (color class) from $\mathbb{A}$ containing $v_j$". Moreover, we consider the following facts:\\
$\bullet$ no vertex of $H_j$ receives the color $f(v_j)$, otherwise there would be a vertex of $H_j$ adjacent to at least two vertices with the same color $f(v_j)$ since $H$ has no isolated vertices; and\\
$\bullet$ none of the colors, assigned to the vertices of those open packing sets (color classes) from $\mathbb{A}$ containing the neighbors of $v_j$, can be assigned to the vertices of $H_j$.\\
The above argument shows that $G\odot H$ cannot be injectively colored with $\chi_{i}(G)$ colors. Hence, $\chi_{i}(G\odot H)>\chi_{i}(G)$.

Obviously, $f$ assigns at most $\Delta(G)+1$ colors to the vertices in $N_{G}[v_i]$ for each $1\leq i\leq n$. We now choose a vertex $v_{i}$ and let it be in $A_{k}$. If $|V(H)|\leq \chi_{i}(G)-\deg_{G}(v_i)-1$, then we assign $|V(H)|$ colors $r_1,\dots,r_{|V(H)|}\in\{1,\dots,\chi_{i}(G)\}$ to the vertices $u_{i1},\dots,u_{i|V(H)|}$ such that $N_{G}[v_{i}]\cap A_{r_j}=\emptyset$, with $1\le j\le |V(H)|$, similarly to Case $1$ (indeed, $f$ assigns at most $\chi_{i}(G)$ colors among $\{1,\cdots,\chi_{i}(G)\}$ to the vertices in $V(H_i)\cup V(G)$). Otherwise, we deal with the following two possibilities.\vspace{1.5mm}\\
\textit{Subcase 2.2.1.} $\chi_{i}(G)\in \{\deg_{G}(v_i),\deg_{G}(v_i)+1\}$. In such a situation, we assign $|V(H)|$ new colors $1',\cdots,|V(H)|'$ to the vertices $u_{i1},\dots,u_{i|V(H)|}$, respectively. In fact, $f$ has used $|V(H)|+\Delta(G)$ or $|V(H)|+\Delta(G)+1$ colors in order to injectively color the vertices in $V(H_i)\cup V(G)$.\vspace{1.5mm}\\
\textit{Subcase 2.2.2.} $\chi_{i}(G)>\deg_{G}(v_i)+1$. We then assign $k(i)=\chi_{i}(G)-\deg_{G}(v_i)-1$ colors $r_1,\dots,r_{k(i)}\in\{1,\dots,\chi_{i}(G)\}$ to the vertices $u_{i1},\dots,u_{ik(i)}$ such that $N_{G}[v_{i}]\cap A_{r_j}=\emptyset$ with $1\le j\le k(i)$, and $t(i)=|V(H)|-k(i)$ new colors $1',\cdots,t(i)'$ to the vertices $u_{i(k(i)+1)},\dots,u_{i|V(H)|}$, respectively.

Iterating this process for all vertices $v_i$ with $\chi_{i}(G)>\deg_{G}(v_i)+1$, we observe that $f$ assigns at most
$$\chi_{i}(G)+\max_{i}\{t(i)\}=\chi_{i}(G)+|V(H)|-\chi_{i}(G)+\Delta(G)+1=|V(H)|+\Delta(G)+1$$
colors in order to injectively color the vertices in $V(H_i)\cup V(G)$.

Notice that the extension of $f$ given in Subcases 2.2.1 and 2.2.2 results in an injective coloring of $G\odot H$. From this fact, we deduce that $\chi_{i}(G\odot H)\leq|V(H)|+\Delta(G)+1$. On the other hand, $\chi_{i}(G\odot H)\geq \Delta(G\odot H)=|V(H)|+\Delta(G)$ by the inequality (\ref{EQ1}). Therefore, $\chi_{i}(G\odot H)$ equals either $|V(H)|+\Delta(G)$ or $|V(H)|+\Delta(G)+1$.\vspace{0.5mm}

Altogether, the arguments above show that $\chi_{i}(G\odot H)\in\big{\{}\chi_{i}(G),|V(H)|+\Delta(G),|V(H)|+\Delta(G)+1\big{\}}$.
\end{proof}

We conclude this section with remarking that $\chi_{i}(G\odot H)$ assumes all three values given in Theorem \ref{Cor} depending on our choices for $G$ and $H$. To see this, let $G=K_{r}\square K_{s}$ for two integers $r,s\geq3$. It is clear that any injective coloring $f$ of $G\odot K_{2}$ assigns $rs$ colors, say $1,2,\dots,rs$, to the vertices in $V(G)$. Moreover, by assigning two colors from $\{1,2,\dots,rs\}\setminus \{f(u)\mid u\in N_{G}[v]\}$ to the vertices in $N_{G\odot K_{2}}[v]\setminus N_{G}[v]$ for each $v\in V(G)$, we get an injective coloring of $G\odot K_{2}$ with $rs$ colors. Therefore, $\chi_{i}(G\odot K_{2})=rs=\chi_{i}(G)$.

Bre\v{s}ar et al. \cite{bsy} showed that $\chi_{i}(T)=\Delta(T)$ for any tree $T$ on at least two vertices. With this in mind, taking $H$ to be any graph with no isolated vertices, we observe that $T\odot H$ satisfies the assumption given in Subcase 2.2 in the proof of Theorem \ref{Cor}. Hence, $\chi_{i}(T\odot H)\in \{|V(H)|+\Delta(T),|V(H)|+\Delta(T)+1\}$. On the other hand, any injective coloring of $T$ with $\Delta(T)$ colors can be extended to an injective coloring of $T\odot H$ with $|V(H)|+\Delta(T)$ colors by assigning $|V(H)|$ new colors to the vertices of $H_{1},\dots,H_{|V(T)|}$. This leads to $\chi_{i}(T\odot H)=|V(H)|+\Delta(T)$.

Finally, we observe that for any graph $H$ with no isolated vertices, $\chi_{i}(K_{n}\odot H)=n+|V(H)|=|V(H)|+\Delta(K_{n})+1$ for $n\geq3$.


\section*{Acknowledgments}

B. Samadi and N. Soltankhah have been supported by the Discrete Mathematics Laboratory of the Faculty of Mathematical Sciences at Alzahra University.


\end{document}